\documentclass[11pt]{article}
\usepackage{geometry}                
\usepackage{graphicx}
\usepackage{url}
\usepackage{amssymb}
\usepackage{epstopdf}
\usepackage{amsthm}
\usepackage{amsmath}
\newtheorem{theorem}{Theorem}
\newtheorem{lemma}{Lemma}
\newtheorem{remark}{Remark}

\newcommand{\R}{{\mathbb R}}
\newcommand{\Z}{{\mathbb Z}}
\newcommand{\N}{{\mathbb N}}

\newcommand{\ch}{{\bf 1}}

\newcommand{\ol}[1]{\overline{#1}}
\newcommand{\lr}[1]{\lfloor#1\rfloor}
\newcommand{\eqd}{\overset{\footnotesize{d}}{=}}
\title{An Empirical Process Interpretation of a Model of Species Survival}
\author{Iddo Ben-Ari\footnote{This work was partially supported by a grant from the Simons Foundation (\#208728 to Iddo Ben-Ari)} \\~  Department of Mathematics\\ University of Connecticut\\ 196 Auditorium Rd \\ Storrs, CT 06269-3009\\~\\
\url{iddo.ben-ari@uconn.edu}}
\begin{document}
\maketitle
\begin{abstract}
We study  a model of species survival recently proposed by Michael and Volkov. We interpret it as a variant of  empirical processes, in which the sample size is random and  when decreasing, samples of smallest numerical values are removed. Micheal and Volkov proved that the empirical distributions converge to the sample distribution conditioned not to be below a certain threshold.  We prove a functional central limit theorem for the fluctuations. There exists a threshold above which the  limit process is Gaussian with variance bounded below by a positive constant,  while at the threshold it is half-Gaussian. 
 \end{abstract}
 
\section{Introduction and Statement of Results}
We study a generalization of the Guiol Machado Schinazi (GMS) model \cite{GMS} \cite{BMR} \cite{GMS_2} that was recently proposed and analyzed  by Michael and Volkov \cite{myvol}. \\
 
 The model could be  viewed as describing an ecosystem whose population size is given by a simple  Markov chain on $\Z_+$.  Each member of the ecosystem has a random  ``fitness" assigned at birth. When the population size decreases, the ``least fit" members are eliminated.  The population size process is driven by an IID sequence of $\Z$-valued random variables, $(I_n:n\in\N)$.  Starting with population size equal to  $0$ at time $n=0$, at each time $n\in\N$, the population increases by $I_n$ if $I_n\ge0$, or decreases by the  minimum between the present population size and $|I_n|$ if $I_n <0$. In the original  GMS model, $I_n \in \{-1,1\}$, that is the population size is  modeled by a birth and death chain. \\ 
 
 When $I_n \equiv 1$ for all $n$,  the ecosystem at time $n$ consists of $n$ IID samples from a $U[0,1]$ distribution, hence the immediate connection to empirical processes. In this model there are two additional ingredients. The first  is the randomness of the sample size. This is not new, e.g. Pyke \cite{pyke} and also, in a closely related context, Asmussen \cite{asmussen}. The  former paper studies empirical processes in which the sample size is random and obeys a law of large numbers with positive speed. In the latter, among other things, the author proves a scaling limit for empirical distributions  corresponding to  sample size given by an increasing sequence of stopping times which are  infinite with positive probability, conditioned to be finite. The second ingredient of the present model, and which appears to be new, is the mechanism according to which samples are discarded when the sample size decreases.  This mechanism is responsible for criticality : The empirical distribution converges to the sample distribution conditioned not to drop below a certain threshold, in contrast to the classical Glivenko-Cantelli theorem, where the empirical distributions converge to the sample distribution.  A result of the same spirit holds in \cite{asmussen}, due the conditioning. Furthermore, as we will show below, fluctuations from this distribution scale to a process which is discontinuous  at  the critical threshold. The  process is Gaussian except at the critical threshold, where it is  half-normal  (the absolute value of a centered normal).  This deviates from the ``classical" Brownian Bridge scaling for empirical processes (e.g. \cite[Theorem 14.3]{Billin}), which is also the scaling limit in \cite{pyke} and \cite{asmussen}. \\
   
 We turn to a formal description of the model. Let $I,I_1,\dots$ be an IID sequence of $\Z$-valued random variables. We define the population size process $X$ by letting   
$$ X_0:=0,~ X_{n+1} : = X_n + \max (I_{n+1}, -X_n),~n\in\Z_+. $$ 
This inductive formula gives the  waiting time of the $n+1$-th customer  in a G/G/1 queue, with $I_{n+1}$ interpreted as  the difference between the  service time of $n$-th customer and interarrival time between $n$-th and $n+1$-th customer. However  the main object of interest in the present model is the additional and intrinsic fitness structure, which does not translate naturally into queuing theory. \\
 
 For  $f\in [0,1]$ and $n\in \Z_+$, let $L_n (f)$ denote the number of members of the population at time $n$  whose fitness does not exceed $f$,  and write $L(f):=(L_n (f):n\in \Z_+)$ for the corresponding process.  Here is an explicit construction. For $n\in \Z_+$, let $S_{n,+} :=\sum_{0<j\le n} (I_j)_+$ and similarly, $S_{n,-}:=\sum_{0<j\le n} (I_j)_-$, where we here and henceforth we  convene that  summation over an empty index set has sum  $0$, and  for a  real number $x$, we define   $x_+:=\max(x,0),~x_-:=(-x)_+ = -\min(x,0)$.  Let $U,U_1,\dots$ be an IID sequence sampled from a $U[0,1]$ distribution.  For $f\in [0,1]$ and $n\in\Z_+$, let  
$$  C_{n+1}  (f) := \sum_{S_{n,+} <  j\le   S_{{n+1},+}}\ch_{[0,f]}(U_j),$$
 that is, $C_{n+1}(f)$ represents the number of members of population born at time  $n+1$ and whose fitness does not exceed $f$.  As in the construction of $X$, we let 
\begin{equation} 
  \label{eq:Ln} 
  L_0(f) :=0,~L_{n+1}(f) := L_n(f) + \max (C_{n+1}(f) - (I_{n+1})_-, - L_n (f)).
\end{equation}  
  Since $C_{n+1}(f)$ and $(I_{n+1})_-$ are independent of $L_0,\dots,L_n$, it follows that for each fixed $f$, $L(f)$ is a $\Z_+$-valued Markov chain. Note that $L_n (1)=X_n$. \\
  
  We now provide an alternative construction which will be frequently utilized in the proofs. For $f \in [0,1]$, let $S(f):=(S_n (f):n\in \Z_+)$ denote the  process consisting of the partial sums of IID random variables : 
   $$ S_n (f) : = \sum_{0<j\le n} (C_j (f) - (I_j)_-)=\sum_{0<j\le S_{n,+}} \ch_{[0,f]}(U_j)-S_{n,-},~n\in \Z_+. $$ 
  We also define the corresponding sequence of running minima, 
   $$ M_0(f) := 0,~M_{n+1}(f) = \min (M_n(f),S_{n+1}(f)),~n\in \Z_+.$$ 
The following identity is well-known, and we provide a proof only for  convenience : 
\begin{lemma} 
\label{lem:LS}
For all $f \in [0,1]$ and $n\in\Z_+$, 
  $L_n(f) = S_n (f)- M_n(f)$  
\end{lemma} 
Thus, the family of processes $(L(f):f \in [0,1])$ consists of coupled reflected random walks. 
\begin{proof}
      The claim is clearly true for $n=0$. We continue by induction,  
   \begin{align*}
    L_{n+1} (f) &= S_{n} (f)-M_{n} (f) -\min  (-(C_{n+1}(f) - (I_{n+1})_-), S_n(f) - M_n (f)) \\
     & = S_n (f) - \min ( -(C_{n+1}(f) - (I_{n+1})_-)+M_n(f),S_n (f))\\
      & = S_n(f) +(C_{n+1}(f) - (I_{n+1})_-) - \min (M_n(f),S_n(f) + C_{n+1}(f) -(I_{n+1})_-)\\
     & =S_{n+1}(f) - M_{n+1}(f).      
     \end{align*}
\end{proof} 


Since we are interested in scaling limit for the empirical distribution function of the fitnesses, we wish to  work under assumptions that guarantee $\lim_{n\to\infty} X_n = \infty$ a.s. This is why we make the following assumption :  
 \begin{equation}
 \label{eq:f_cexists} 
 E I_- < E I_+ \le \infty. 
 \end{equation}
 Under assumption \eqref{eq:f_cexists} we  define a critical threshold  for the fitness, $f_c$ : 
  $$ f_c : = \frac{E I_-}{E I_+} \in [0,1).$$
We will focus on behavior of the ecosystem, restricted to  the  fitness interval $[f_c,1]$, where a central limit theorem holds. As is not hard to show, for every $f<f_c$, $L(f)$  is positive recurrent and converges to its invariant distribution without centering and scaling. \\

To present our results, we define the empirical distribution function $\hat F_n$, 
 $$ \hat F_n (f) := \begin{cases} \frac{L_n (f)}{X_n} & X_n >0\\ 0; &  \mbox{otherwise.}\end{cases} $$ 
    We also let $F$ denote the distribution function of the $U[f_c,1]$ law. That is, $$F(f) = \frac{ (\min (f,1)-f_c)_+}{1-f_c},$$
    and let
    $$\hat \Delta_n := \hat F_n -F.$$ 
  
Our first result is an analog to the  Glivenko-Cantelli Theorem. 
  \begin{theorem} Assume condition \eqref{eq:f_cexists}. Then 
\label{th:glivenko}
  $$ \lim_{n\to\infty}   \|  \hat \Delta_n   \|_\infty=0,\mbox{ a.s.}$$ 
\end{theorem}
This result follows directly  from the work of  Michael and Volkov  and is    equivalent  to \cite[Theorem 4-(b)]{myvol}.  The result in  \cite{myvol} is  stated in terms of coupling between the ecosystem and an IID sample drawn from a $U[f_c,1]$ distribution. We provide an alternative proof based on the Glivenko-Cantelli Theorem and the law of large numbers. \\

We turn to the Functional Central Limit Theorem. For the remainder of the section we impose the moment condition 
   \begin{equation} 
   \label{eq:mom} 
    E (I^2)<\infty. 
    \end{equation} 
    We first introduce the processes which appear in the statement of our result. Let $W_1,W_2$ be independent copies of standard Brownian motion,  and let $\mbox{Br}_1$ denote the Brownian bridge associated with $W_1$. That is,  
   $$\mbox{Br}_1 (f) : = W_1 (f) - f W_1 (1),~f \in [0,1].$$ 
   We also define the process  $\tilde W_1$   on $[0,1]$ as follows. When $f_c=0$ (equivalently, $EI_-=0$),  we let  $\tilde W_1\equiv 0$. Otherwise,  let $U\sim   U[0,1]$ be independent of $W_1$ and $W_2$.  
    For $t \in [0,1)$, let $A_t:= ([0,t]+U) \mod 1$, and let $\tilde A_t := f_c + (1-f_c)A_t$. Define 
    $$ \tilde W_1(t) :=  \frac {1}{\sqrt{f_c (1-f_c)}}  \left ( (1-f_c) W_1(f_c t) - f_c \int  \ch_{\tilde A_t}(s)d W_1 (s)\right),~t \in[0,1], $$
     and  
      $$  Y_\infty(t):=\tilde \sigma_1   \tilde W_1(t)+ \sigma_2 W_2 (t),~t\in [0,1].$$ 
       where 
       $$\tilde \sigma_1 := \sqrt{f_c(1-f_c) E I_+},\mbox{ and } \sigma_2 := \sigma ( f_c I_+ - I_-) = \sqrt{f_c^2 \sigma^2(I_+) + 2 f_c E I_+ E  I_- + \sigma^2(I_-)},$$
        and where here and henceforth,  $\sigma(Z)$ denotes the  standard deviation of the square integrable  random variable $Z$. We record the following. 
        \begin{lemma}  
Assume that conditions \eqref{eq:f_cexists},\eqref{eq:mom} hold, and that $f_c>0$. 
\label{lem:W4BM} 
 Then $\tilde W_1$ is standard Brownian motion on $[0,1]$. 
\end{lemma} 
\begin{proof}
 Observe that $|\tilde A_t|=t(1-f_c)$, independently of $U$, and that $\tilde A_s \subset \tilde A_t$ if $s<t$.  As an immediate consequence,  the finite-dimensional distributions of $\tilde W_1$ conditioned on $U$ coincide with those of $W_1$, and since $\tilde W_1$ is a continuous processes, the result follows. 
  \end{proof}
For  a compact interval $I$, we denote by $D(I)$  the Skorohod space of  real-valued cadlag functions on the compact interval $I$,  equipped with the $J_1$-topology.   Let $\Psi : D ([0,1])\to \R$ be the function 
 $$\Psi (x) :=x(1) - \inf_{0\le t \le 1} x(t).$$
Here is our main result. 
 \begin{theorem}
 \label{th:CLT}  
Assume that \eqref{eq:f_cexists} and \eqref{eq:mom} hold. Then for any $f_0 \in (f_c ,1]$ we have   
  $$\sqrt{n} \left (   \hat {\Delta}_n |_{[f_0,1]},\hat F_n (f_c) \right)\Rightarrow \frac{1}{E I_+} 
  \left ( \sqrt{E I_+} \mbox{Br}_1 + (1-F)  \sigma_2 W_2(1), \Psi(Y_\infty)\right)$$ 
  in $D([f_0,1])\times \R$. 
 \end{theorem}
 \begin{remark}
 \label{rem:open} 
  The construction of the Skorohod space  $D ([0,\infty))$, given for example in \cite[Section 16, p. 166--179]{Billin} could be repeated {\it mutatis mutandi} to obtain the Skorohod space $D((f_c,1])$. This results in a complete separable metric space, with the property that weak convergence is characterized by convergence of the restriction to $D([f_0,1])$ for every $f_0 \in (f_c,1]$.   \end{remark} 
  For an explicit formula for the one-dimensional marginals, let 
$$g(f) : =\frac{1}{E I_+} \sqrt{  f (1-f)E I_+    + \left (\frac{1-f}{1-f_c}\right)^2 \left ( f_c^2 E (I_+^2) +2 f_c EI_+ E I_- + E (I_-^2)\right)}.$$
 and let $Z$ denote a standard normal random variable. Then in terms of the marginals, the theorem states that   
  $$
    \sqrt{n}   \hat \Delta_n (f) \Rightarrow g(f) \begin{cases}  Z  & f>f_c \\ 
    |Z|  & f=f_c,
     \end{cases}  
  $$
   where the limiting distribution for $f=f_c$ is due to the well-known identity 
    $\Psi(W_1)  \eqd |W_1(1)|$,  (e.g.  \cite[Corollary 2.23]{Kalen}). Since for every $f\in [0,f_c)$, $L(f)$ is positive recurrent, and $f \to L_n (f)$ is a nonnegative nondecreasing function, it easily follows that $(\frac{L_n}{\sqrt{n}} |_{[0,f_c)}:n\in\N)$ converges uniformly on compacts to $0$ in probability and then, the  same holds for $(\sqrt{n}\hat \Delta_n|_{[0,f_c)}:n\in\N_+)$.  In light of the above,  the sequence  of right-continuous processes $(\sqrt{n}\hat  \Delta_n (\cdot):n\in\N)$  converges in some appropriate sense  to a process on $[0,1]$ which is identically $0$ on $[0,f_c)$, is nondegenerate half-normal at $f_c$, and a continuous Gaussian process on $(f_c,1]$,  whose right limit at $f_c$ is nondegenerate. In particular, the sequence is not-tight in $D([0,1])$, and neither is its restriction to $D([f_c,1])$. \\
    
Note that if we relax \eqref{eq:f_cexists} and allow $E I_- = E I_+$, then the ``scalable" fitness interval $[f_c,1]$  reduces to a point $\{1\}$, because for any $f<1$, $L(f)$ is positive recurrent. It follows from the definition of $\hat F_n$, that $\hat F_n (1)$ is the indicator of the event $\{X_n >0\}$. However, when $E I_-=E I_+<\infty$, this event occurs infinitely often, a.s. However,  $X_n = S_n (1) - M_n (1)$ and $S_n(1)=\sum_{0<i\le n} I_j$, and it follows from Donsker's theorem that when $E I_-=EI_+$ and $E (I^2)<\infty$, $X_n /\sqrt{n} \Rightarrow \Psi( \sigma(I) W_1)$.   \\

We  comment that if $D$ is any distribution function,  we may consider the ecosystem obtained as above but with fitnesses sampled from the distribution $D$ rather than the $U[0,1]$ distribution. If $D$ has atoms then a.s. multiple elements with the same fitness will exist, and therefore we impose the additional condition that when when population size decreases the least fit elements are removed, and in case there are more elements with identical fitness than we need to remove, we choose those to be removed by the index of the random variable which determined their fitness, in increasing order. It is easy to see that the number of elements with fitness not exceeding $f$ at time $n$ is then given by $L_n (D(f))$  reducing  the analysis to the uniform fitness model above. \\

As a final remark,  we observe that the  ecosystem has an interpretation as an urn system.  Let $\mu$ be a probability on the subsets of $\N$, and suppose that  we have a system of urns labeled by $\N$ (only urns in the support of $\mu$ will play a role here). We begin with all urns empty at time $0$. Then at each time $n\in\N$,  we add $I_n$ balls to the system if $I_n\ge 0$ or attempt to remove balls from the system otherwise. If $I_n \ge 0$, we add $I_n$ balls, with each  ball placed in urn $k$ with probability $\mu (\{k\})$, independently of the others. If $I_n <0$, then we remove the minimum between the number of the balls in the system and $|I_n|$, starting from urn $1$, in increasing order, and ignoring empty urns. The combined number of balls in urns $1,\dots,k$ at time $n$, is then given by $L_n (\mu (\{1,\dots,k\}))$.  
 \section{Proofs}
\subsection{Proof of Theorem \ref{th:glivenko}}
\begin{lemma} 
\label{lem:limits}
Assume condition \eqref{eq:f_cexists}. If $ E I_+<\infty$, then 
 \begin{enumerate} 
 \item
 $  \lim_{n\to\infty} L_n (f)/n = (f-f_c)_+ E I_+$, a.s. 
 \item $\lim_{n\to\infty} M_n (f)/n = -(f-f_c)_- E I_+$, a.s. 
 \end{enumerate} 
  \end{lemma}
  \begin{proof} 
  Observe that  $E S_n (f) = f E I_+ - E I_- = (f-f_c)E I_+$. Thus,  $S_n(f)$ is transient to $+\infty$ if $f>f_c$, null recurrent if $f=f_c$ and transient to $-\infty$ if $f<f_c$. In particular, $\inf_{n\in \N} M_n (f) >-\infty$ a.s. or $=-\infty$ a.s. according to whether $f>f_c$ or $f\le f_c$. Since by Lemma \ref{lem:LS},  $L_n (f) = S_n (f)- M_n(f)$, it follows from the Law of Large Numbers that for $f>f_c$, 
  $$ \lim_{n\to\infty} L_n (f) / n = (f-f_c) E I_+,\mbox{ a.s.}$$
   But since $L_n (\cdot)$ is nonnegative and monotone, it then follows that $ L_n(f)/n \to 0$ a.s. for all $f\le f_c$. \\
    As a result, for $f\le f_c$, 
    $$\lim_{n\to\infty} M_n (f)/n = \lim_{n\to\infty} S_n (f)/n= (f-f_c) E I_+, \mbox{ a.s.},$$
    completing the proof. 
\end{proof}  
\begin{proof}[Proof of Theorem \ref{th:glivenko}]
It follows from Lemma \ref{lem:limits}-(1) that  $\lim_{n\to\infty} X_n =\infty$ a.s. In particular, by taking $n$ large enough  we may assume $X_n >0$. Let  
 let $\hat Q_k(f):= \frac{\sum_{0<j \le k} \ch_{[0,f]}(U_j) - k f}{k}$. Then 
 $$\hat F_n (f) = \frac{1}{X_n} \left ( S_{n,+} \hat Q_{S_{n,+}}(f)+ f S_{n,+}-S_{n,-}-M_n (f) \right),$$
  and we have 
  \begin{align}
  \nonumber  
  \hat \Delta_n (f)&= \frac{S_{n,+}}{X_n} \hat Q_{S_{n,+}} (f) +\frac{ f_c S_{n,+} - S_{n,-}}{X_n} + (f-f_c) \frac{S_{n,+}}{X_n} - \frac{ (f-f_c)_+ }{1-f_c}- \frac{M_n (f)}{X_n}  \\
   & =   \frac{S_{n,+}}{X_n} \hat Q_{S_{n,+}}(f) + \frac{ f_c S_{n,+} - S_{n,-}}{X_n} \\
  \label{eq:above} &~~~~~+ (f-f_c) \left (\frac{ S_{n,+}}{X_n} - \frac{1}{1-f_c}\right)-\frac{(f-f_c)_-}{1-f_c}-\frac{M_n(f)}{X_n}.
 \end{align} 
  Since $S_{n,+}\to \infty$ a.s., it follows from the Glivenko-Cantelli  Theorem that $\lim_{n\to\infty } \sup_{f\in [0,1]} |\hat Q_{S_{n,+}}(f)|=0$.  
  Assume now that $E I_+<\infty$. Then by Lemma \ref{lem:limits}, $\lim_{n\to\infty} X_n/n = (1-f_c) E I_+$, a.s. It then follows from the Law of Large Numbers that 
   $$ \lim_{n\to\infty} \frac{S_{n,+}}{X_n} = \frac{1}{1-f_c }\mbox{ and } \lim_{n\to\infty} \frac{f_c S_{n,+}-S_{n,-} }{X_n} =0,\mbox{ a.s.}$$
    In addition,  $\sup_{f \in [f_c,1]} |M_n (f) |\le |M_n (f_c)|$, and by Lemma \ref{lem:limits}-(2), $\lim_{n\to\infty} |M_n (f_c)|/n =0$ a.s.    Thus it follows from  \eqref{eq:above} that  $\lim_{n\to\infty} \sup_{f\in[ f_c,1]} |\hat \Delta_n (f)| = 0$, a.s. Next, observe that  since $F|_{[0,f_c]}\equiv 0$, 
    $ \sup_{f \in [0, f_c]} |\hat \Delta_n (f)| = \sup_{f\in [0,f_c]} L_n (f)/ X_n= L_n(f_c)/ X_n$. But by Lemma \ref{lem:limits}-(1), $\lim_{n\to\infty} L_n (f_c)/X_n=0$, a.s. This completes the proof for the case $E I_+<\infty$. \\ 
  
   We now assume  $E I_+=\infty$. Here $f_c=0$ and we may rewrite \eqref{eq:above}  as 
   \begin{equation*}
   \label{eq:representation} 
    \hat \Delta_n (f):=   \frac{S_{n,+}}{X_n} \hat Q_{S_{n,+}}(f) - \frac{S_{n,-}}{X_n} + f  \left (\frac{ S_{n,+}}{X_n} - 1 \right)-\frac{M_n (f)}{X_n}.
    \end{equation*}
   Since $ E I_-<\infty = E I_+$, we have   $\lim_{n\to\infty} S_{n,-}/X_n =0$ a.s., and 
     $$ \frac{ S_{n,+}}{X_n} \le   \frac{ S_{n,+}}{S_{n,+}-S_{n,-} }= \frac{ 1}{1- \frac{S_{n,-}}{S_{n,+}}}.$$
      Therefore $\lim_{n\to\infty} \frac{S_{n,+}}{X_n} =1$ a.s. Finally, $\sup_{f\in [0,1]} |M_n (f)|=|M_n(0)|= S_{n,-}$.    Thus,      $ \lim_{n\to\infty} \sup_{f \in [0,1]} | \hat \Delta_n (f ) |=0$, a.s.
        \end{proof}
  \subsection{Proof of Theorem \ref{th:CLT}} 
  We first wish to explain the problem and how one can solve it. Consider  \eqref{eq:above} from the proof of Theorem \ref{th:glivenko} above. Multiplying both sides by $\sqrt{n}$, it follows from the Glivenko-Cantelli theorem that the first term on the right-hand side converges to a Brownian Bridge multiplied by some constant, and it is not hard to show that jointly, the  sum of the second and third term converges to an independent standard normal random variable multiplied by an  affine function.  The fourth term vanishes on $[f_c,1]$, and the last term, $-\sqrt{n} M_n (f) /X_n= \frac{M_n/\sqrt{n}}{X_n/n}$ converges uniformly to $0$ on $[f_0,1]$ for every $f_0 \in (f_c,1]$. This argument leads to a proof of the  convergence claim for the first component in Theorem \ref{th:CLT}, $\sqrt{n}\hat \Delta_n |_{[f_0,1]}$. However, the joint convergence statement in the theorem requires more. This is because at $f_c$ the fifth term depends on the past, $(S_k(f):k\le n)$, and  scales to a nondegenerate random variable. In order to take this into account we will consider the centered  processes $(\ol{S}_n (f): f \in (f_c,1])$ and $(\ol{S}_{[nt]}(f_c):t \in [0,1])$, where
 $$\ol{S}_n (f) := S_n (f) - E S_n (f) = S_n (f) - (f-f_c) n E I_+.$$
We will prove that the  pair of processes scales jointly to some limiting process in the product space. We will then express  $\sqrt{n} \hat \Delta_n $ as a function of the pair and apply the convergence result to deduce the theorem. The core of our proof is identifying the covariance structure of the limit process in the product space, not the actual convergence. In fact, the methods of Bickel and Wichura \cite{Bi_Wi} allow to prove convergence of the two-parameter process $(\frac{\ol{S}_{[nt]}(f)}{\sqrt{n}}:(t,f)\in[0,1]^2)$, but this is not needed in our work. \\

We begin with a simple tightness statement. 
 \begin{lemma} 
 \label{lem:tight}
For any $f_0 \in [0,1]$, the laws of  $(\frac{\ol{S}_n }{\sqrt{n}}|_{[f_0,1]}:n\in\N)$  in $D([f_0,1])$ are tight. 
\end{lemma} 

 \begin{proof}
  Let $Q_k(f) := \frac{ \sum_{j=1}^k \ch_{[0,f]}(U_j) - k f}{\sqrt{k}}$. Then by the Central Limit Theorem for empirical processes,  \cite[Theorem 14.3]{Billin}, $Q_k\Rightarrow \mbox{Br}_1$ in $D[0,1]$. Let $T:D([0,1])\to D ([f_0,1])$ denote the restriction mapping $T x = x|_{[f_0,1]}$. Note that $T$ is not continuous, but  $T$ is measurable. Since $\mbox{Br}_1$ is a continuous process, it follows from the Mapping Theorem \cite[Theorem 2.7]{Billin} that $Q_k | _{[f_0,1]}\Rightarrow{\mbox{Br}_1}|_{[f_0,1]}$ in $D([f_0,1])$. In particular,   the restricted processes $(Q_k | _{[f_0,1]}:k\in \N)$ are tight in $D([f_0,1])$.  
  For $A\subset D([f_0,1])$ and $c \in \R$, let  $c A:= \{ cx : x \in A\}$. If $A$ is compact and $I\subset \R$ is compact, then $A_I:= \cup_{c \in I} c A$ is compact in $D([f_0,1])$. Indeed, if $(c_n x_n:n \in\N)$ is any sequence in $A_I$, such that $c_n\to c$ in $I$ and $x_n \to x$ in $D ([f_0,1])$, then for any sequence $(\lambda_n:n\in\N)$ of homeomorphisms of $[f_0,1]$ onto $[f_0,1]$ satisfying
   $$ \| \lambda_n (t) -t\|_\infty \to 0,~ \|x_n - x \circ \lambda_n \|_\infty \to 0,$$
    we have 
  $$ \|c_n x_n - (cx) \circ \lambda_n \|_\infty \le |c_n -c| \|x_n\|_\infty + |c| \|x_n -x \circ \lambda_n \|_\infty \to 0.$$
   Thus, every sequence in $A_I$ has a subsequence converging  in $D([f_0,1])$ to a limit in  $A_I$.
  Let $R_n(f):= \frac{f S_{n,+}- S_{n,-} - (f-f_c)n E I_+}{\sqrt{n}}$. 
  It is easy to see that the  laws of  $(R_n|_{[f_0,1]}:n\in\N)$  as processes in $C([f_0,1])$ are tight. Indeed, 
  $P( |R_n(f_0)|\ge a) = P( \frac{|(f_0 S_{n,+} -S_{n,-})-(f_0-f_c) n E I_+|}{\sqrt{n}} \ge a)$, therefore it follows from the Central Limit Theorem that 
  \begin{equation} 
  \label{eq:one_pt} 
  \lim_{a\to\infty} \lim_{n\to\infty} P(|R_n (f_0)|\ge a)=0.
\end{equation}
  Furthermore
    $$ P ( \sup_{f \in [f_0,1],u <\delta} |R_n (f) - R_n ((f+u) \wedge 1)|\ge  \epsilon ) \le
     P (\frac{  |\delta| |S_{n,+}- n E I_+|}{\sqrt{n}} \ge \epsilon),$$
     and again the by the Central Limit Theorem, 
     \begin{equation}
     \label{eq:oscillations} 
   \lim_{\delta \to 0} \lim_{n\to\infty} P ( \sup_{f \in [f_0,1],u <\delta} |R_n (f) - R_n ((f+u) \wedge 1)|\ge  \epsilon ) =0.
   \end{equation} 
   By \cite[Theorem 7.3]{Billin}, \eqref{eq:one_pt} and \eqref{eq:oscillations}  guarantee that the laws of $(R_n|_{[f_0,1]}:n\in\N)$ are tight in $C([f_0,1])$. 
   Observe that
    $$ \frac{ \ol{S}_n}{\sqrt{n}}    = 
    \sqrt{\frac{S_{n,+}}{n}} Q_{S_{n,+}} + R_n.$$
    In order to simplify notation, in the remainder of the proof we abbreviate $Q_n|_{[f_0,1]}$, $R_n|_{[f_0,1]}$ and $\ol{S}_n|_{[f_0,1]}$ to $Q_n$, $R_n$ and $\ol{S}_n$.  Fix $\epsilon>0$ and let $K_1 \subset D ([f_0,1])$ and $K_2 \subset C([f_0,1])$ be compact sets such that for all $n\in\N$, 
         $$ P( Q_{n} \in K_1)\ge 1-\epsilon,~P(R_n  \in K_2 )\ge 1-\epsilon.$$
 Let $I := \sqrt{E I_+} [1-\epsilon,1+\epsilon]$. By the Law of Large Numbers,
       $ P( \sqrt{\frac{S_{n,+}}{n}} \in I) \ge (1-\epsilon)$ for all $n$ large enough, hence
     $$ P( \frac{ \ol{S}_n}{\sqrt{n}} \in \sqrt{ \frac{S_{n,+}}{n}}K_1+K_2) \le  P ( \frac{\ol{S}_n}{\sqrt{n}} \in (K_1)_I + K_2) + \epsilon .$$
      But, 
    \begin{align*} 
     P( \frac{\ol{S}_n}{\sqrt{n}} \in \sqrt{ \frac{S_{n,+}}{n}}K_1 + K_2 ) & \ge \sum_{k=0}^\infty P( \sqrt{\frac{ S_{n,+}}{n}} Q_{S_{n,+}} \in \sqrt{\frac{ S_{n,+}}{n}}   K_1,R_n \in K_2 , S_{n,+} = k) 
     \\& \ge
      \sum_{k=0}^\infty P( Q_k \in K_1, R_n \in K_2 | S_{n,+}=k)P( S_{n,+} =k) \\
      & = \sum_{k=0 }^\infty P(Q_k \in K_1) P( R_n \in K_2, S_{n,+}=k) 
      \ge (1-\epsilon)^2\ge 1-2\epsilon.
      \end{align*}
      Therefore, for all $n$ large enough, $P(\frac{\ol{S}_n}{\sqrt{n}} \in (K_1)_I + K_2) \ge 1-3\epsilon$. It remains to prove that $(K_1)_I+K_2$ is compact in $D([f_0,1])$. Let $(x_n+y_n:n\in\N)$ be a sequence in $(K_1)_I+K_2$ with the property $x_n \to x$ in $D([f_0,1])$ and $y_n \to y$ in $C([f_0,1])$. Furthermore, let $(\lambda_n:n\in\N)$ be a sequence of homeomorphisms from $[f_0,1]$ onto $[f_0,1]$ such that
       $$ \| \lambda_n(t)-t\|_\infty\to 0\mbox{ and } \|x_n - x\circ \lambda_n\|_\infty\to 0.$$ 
       Then 
       $$ \|x_n + y_n - (x+y) \circ \lambda_n \|_\infty \le
        \|x_n - x \circ \lambda_n\|_\infty + \|y_n - y\|_\infty+ \|y- y \circ \lambda_n\|_\infty \to 0,$$
         because $y$ is uniformly continuous. Hence any sequence in $(K_1)_I+K_2$ has a subsequence which converges in $D([f_0,1])$ to a limit in $(K_1)_I+K_2$.   

 \end{proof}  
Next we wish  to express the Brownian motion $W_2$ in terms of two other Brownian motions, and define an auxiliary process $X_\infty$. The following is well known. 
 \begin{lemma}
 Let $A,B$ two nonnegative square integrable random variables satisfying $AB =0$, a.s. Then 
  $ \sigma (A) \sigma (B) \ge (E A) (E B)$.
  \end{lemma} 
 \begin{proof}
  The inequality trivially holds if $E A=0$ or $E B=0$. Therefore we will assume that  the right-hand side is strictly positive. Dividing the inequality by $(E A) (E B)$, we obtain the equivalent inequality : 
    $\sigma (A/ E A) \sigma (B /E B) \ge 1$. Therefore there is no loss of generality assuming $ E A = EB =1$. Finally, from Cauchy-Schwarz, 
     $$ 1= E (A -1) (1-B) \le \sigma (A) \sigma (B),$$
      and the claim follows. 
 \end{proof} 
 As a special case, letting $A := I_+$ and $B:=I_-$,  we obtain 
  \begin{equation*}
   \sigma ( I_+) \sigma (I_-) \ge E (I_+) E (I_-),
  \end{equation*} 
and therefore define $\rho \in [0,1]$ by letting 
$$ \rho :=\begin{cases}  \frac{ E(I_+) E (I_-) }{\sigma (I_+)\sigma(I_-)} & \mbox{if } E I_- > 0;\\  
 0&\mbox{otherwise.} \end{cases}$$  
Let $W_2'$ and $W_2''$  be independent copies of standard Brownian motion independent of $W_1$, and define 
$$W_3':= \rho W_2' + \sqrt{1-\rho^2} W_2''.$$
If $f_c =0$, we let $W_2 :=W_2'$. Otherwise, $\sigma_2 >0$ and we let  
$$W_2 := \frac{1}{\sigma_2 } \left ( f_c \sigma(I_+) W_2' + \sigma(I_-)W_3'\right).$$
We observe that when $f,f'\in [0,1]$ and $0\le s\le t \le 1$, we have 
 \begin{align}
 \label{eq:master_covariance} 
 E &\left(( f \sigma(I_+) W_2'(s)+\sigma(I_-)W_3'(s) )(f'\sigma(I_+) W_2'(t) + \sigma(I_-)W_3'(t) )  \right) \\
 \nonumber
 & =
 s \left(  ff' \sigma^2 (I_+) +\sigma^2 (I_-) + (f +f' )\sigma(I_+) \sigma(I_-) \rho  \right )\\ 
 \nonumber
&=  s \left ( ff' \sigma^2 (I_+) + \sigma^2 (I_-) + (f+f') E I_+ E I_- \right).
 \end{align} 
 When $f=f'=f_c$ and $s=t$, the right-hand side is equal to $\sigma_2^2$. Therefore $W_2$ is indeed a standard Brownian motion. 
We also let 
 $$ X_\infty(f):= \sqrt{E I_+} \mbox{ Br}_1 (f) + f \sigma (I_+) W_2'(1) + \sigma (I_-) W_3'(1),~f \in [0,1].$$ 
\begin{lemma}
 Let $0\le  s\le t \le 1$ and $0 \le f \le f'\le 1$. Then we have 
 \label{lem:covariance} 
 \begin{enumerate}
 \item $E \left (  X_\infty (f) X_\infty (f')\right ) =f(1-f') E I_+  + f f' \sigma^2  (I_+) + \sigma^2 (I_-) +   (f+f') E I_+ E I_-$. 
 \item $ E \left (Y_\infty (s) Y_\infty (t) \right) = s E (X_\infty (f_c)^2)$. 
 \item $ E \left ( X_\infty (f) Y_\infty (t)\right) = t E \left (X_\infty (f) X_\infty (f_c)\right)$. 
 \end{enumerate} 
\end{lemma} 
\begin{proof} 
The first assertion   is a consequence of \eqref{eq:master_covariance}. The second follows from the definition of $Y_\infty$ and the first. We prove the third. The left-hand side is equal to 
   $$ \underset{(*)}{\underbrace{\sqrt{f_c(1-f_c)}E I_+ E \left(  \mbox{Br}_1 (f) \tilde W_1 (t) \right) }} + \underset{(**)}{\underbrace{
     E \left ( \left ( f \sigma(I_+)  W_2' (1) +\sigma (I_-) W_3'(1) \right ) \left ( f_c  \sigma (I_+) W_2' (t) + \sigma(I_-)  W_3'(t)\right ) \right ) }}.$$
By \eqref{eq:master_covariance}, 
$$ (**) =   t\left ( f f_c \sigma^2 (I_+) + \sigma^2 (I_-)  + (f+f_c) E I_+ E I_- \right).$$ 
   To compute $(*)$,  write  $\mbox{Br}_1(f)$ as a sum of three independent random variables, 
\begin{align*}  \mbox{Br}_1(f) &= (1-f) W_1(f)- f (W_1(1)-W_1(f))\\
  & = (1-f) W_1(f_c) + (1-f) (W_1(f)-W_1(f_c))-f (W(1)-W(f)), 
  \end{align*} 
We then have 
       \begin{align*} 
        \sqrt{f_c(1-f_c)} E ( \mbox{Br}_1(f)\tilde W_1(t)) = 
          &
           E \left(  (1-f_c) W_1(f_c t)  (1-f) W_1 (f_c)\right)\\
         & \quad\quad  + E \left ( f_c \int_{[f_c,f]}  \ch_{\tilde A_t} (s) d W_1(s) \times  (1-f) (W_1(f) - W_1(f_c))\right)    \\
         & \quad \quad -E \left ( f_c \int_{[f,1]} \ch_{\tilde A_t}(s) d W_1(s) \times  f (W(1)-W(f))\right) \\
         & = (1-f_c)(1-f) f_c t +f_c (1-f) E|\tilde A_t \cap [f_c ,f]|-f_c f E  |\tilde A_t \cap [f,1]|\\
          & = (1-f_c)(1-f) f_c t + f_c (1-f) (f-f_c) t - f_c f (1-f)t \\
           & = f_c (1-f) t \left ( 1-f_c -(f-f_c) + f\right)\\
            &=  f_c (1-f)t.
         \end{align*}
      Consequently, 
     $$ (*) = (E I_+)f_c (1-f)t,$$
      and the claim then  follows from the first assertion. 
\end{proof}
    Let $Y_n:=(Y_n(t):t\in [0,1])$ be the process defined through  $Y_n(t):= S_{\lr{nt}}(f_c)=\ol{S}_{\lr{nt}}(f_c)$. 
   \begin{lemma}
   \label{lem:findim} 
    The finite dimensional distributions of  $\frac{1}{\sqrt{n}}(\ol{S}_n|_{(f_c,1]},Y_n)$  converge 
    to those of $(X_\infty,Y_\infty)$. 
   \end{lemma} 
   \begin{proof}
 Fix  $N\in\N$.  Let $f_1,\dots,f_N \in (f_c,1]$,  $0=:t_0<t_1< \dots < t_N =1$ and  $\theta_1,\dots,\theta_N, \eta_1,\dots,\eta_N \in \R$. 
   For $n\in\N, f \in [f_c,1]$ and $l=1,\dots,N$, let  $\Delta_{n,l} (f):= \ol{S}_{\lr{n t_l}}   (f) -\ol{S}_{\lr{n t_{l-1}}}(f)$. Observe that 
  $ \ol{S}_n (f_k) = \sum_{l=1}^N \Delta_{n,l} (f_k)$, and that the (function-valued) ``increments" $\Delta_{n,1}(\cdot),\dots,\Delta_{n,N}(\cdot)$ are independent. 
   Let 
    \begin{align*} I_n &:= \sum_{l=1}^N \eta_l \Delta_{n,l}(f_c) + \sum_{k=1}^n \theta_k \ol{S}_n (f_k)\\
   &=  \sum_{l=1}^N \eta_l \Delta_{n,l} (f_c) + \sum_{l=1}^N \sum_{k=1}^N \theta_k \Delta_{n,l}(f_k).
   \end{align*} 
      Let $Z(f):=( \sum_{0<j\le I_+}\ch_{[0,f ]}(U_j)) - I_- - (f-f_c) E I_+$. Observe that $E Z(f)=0$ and $E (\sup_{f} |Z(f)|^2) < \infty$.
       Let $V:=(Z(f_c),Z(f_1),\dots,Z(f_N))$. To simplify notation, we will refer to the $j$-th  entry of $V$ as $V(j-1)$. For example, first entry, $Z(f_c)$, is  $V(0)$, and last, $N+1$-th entry  $Z(f_N)$, is $V(N)$. Let $(V_k:k\in \N)$ be an IID sequence of copies of the random vector  $V$.  We have 
        $$(\Delta_{n,l} (f_c) ,\Delta_{n,l}(f_1), \dots, \Delta_{n,l}(f_N))\eqd \sum_{0< j \le s_l }  V_j,$$
         where $s_l : = \lr{n t_{l}}- \lr {n t_{l-1}}$. 
       Due to the independence of the increments $\Delta_{n,1}(\cdot),\dots,\Delta_{n,N}(\cdot)$ and their above representations as partial sums of IID sequences with finite second moment. Letting   $\theta^{l} := (\theta^l_0,\dots, \theta^l_N)$, where $ \theta^l_0:= \eta_l$ and $ \theta^l_j:= \eta_j$ for $j=1,\dots,N$, we obtain : 
         \begin{align*}
           E( e^{i I_n /\sqrt{n}} ) &= 
            \prod_{l=1}^N \left (  
              E e ^{\frac{i}{\sqrt{n}}\left(  \eta_l \Delta_{n,l} (f_c)  + \sum_{k=1}^N \theta_k \Delta_{n,l}(f_k)\right)}\right)\\
            & = \prod_{l=1}^N E \left ( \left ( E e^{i/\sqrt{n} \left ( \eta_l V(0)+ \sum_{k=1}^N \theta_k V(k)\right)}\right)^{s_l}  \right)\\
          &= \prod_{l=1}^N  E\left(  \left(  E e^{\frac{i}{\sqrt{n}} \theta^l \cdot V}\right)^{s_l}\right). 
         \end{align*} 
   
         Since $E(( \theta^l \cdot V)^2) <\infty$ and $\lim_{l\to\infty} s_l /l = t_l - t_{l-1}$, it follows that
         \begin{equation*}
       \lim_{n\to\infty} E \left ( e^{i I_n /\sqrt{n}} \right)= e ^{-\frac 12 \sum_{l=1}^N (t_l-t_{l-1}) E(( \theta^l \cdot V)^2) }.
        \end{equation*}  
        Therefore to prove the lemma we need to show that 
         \begin{equation}   \label{eq:charfun} \sum_{l=1}^N (t_l-t_{l-1}) E(( \theta^l \cdot V)^2)  = E \left ( \bigl( \sum_{l=1}^N  \eta_l ( Y_\infty (t_l)-Y_\infty(t_{l-1}) )+ \sum_{j=1}^N \theta_j X_\infty (f_j)\bigr) ^2 \right ) .
         \end{equation} 
  For any $0\le j\le k\le N$  and taking $f_0:=f_c$, 
           \begin{align*}
            E \left (   V(j) V(k) \right ) &= E \left ( \bigl(\sum_{0<l\le I_+} \ch_{[0,f_j]}(U_l)   - I_- - (f_j - f_c) E I_+ \bigr) \bigl( \sum_{0<l\le I_+} \ch_{[0,f_k]}(U_l)- I_--(f_k-f_c) E I_+\bigr)\right )  \\
             & = 
             E \left ( \bigl(  \sum_{0<l\le I_+} \ch_{[0,f_j]}(U_l)  \bigr) \bigl ( \sum_{0<l\le I_+} \ch_{[0,f_k]}(U_l) \bigr)\right)  + E (I_-^2) - (f_j-f_c) (f_k-f_c) (E I_+)^2 \\
              & =  f_j  E I_+ + E (I_+^2 - I_+) f_j f_k + E (I_-^2 )- (f_j-f_c)(f_k-f_c)  (E I_+)^2  \\
              & = f_j (1-f_k) E I_+ + f_j f_k \sigma^2 (I_+)  + E (I_-^2 )-f_c^2 (E I_+)^2 + f_c  (f_j +f_k)(E I_+)^2  \\
              & = f_j (1-f_k) E I_++  f_j f_k \sigma^2 (I_+)  + \sigma^2 (I_-)  + (f_j +f_k)E I_+ E I_- \\ 
              & = E ( X_\infty (f_j)  X_\infty (f_k) ).
                         \end{align*}
           Thus,
     \begin{align*}
     \label{eq:inf_composition}
      \sum_{l=1}^N (t_l-t_{l-1}) E \left ( ( \theta^l \cdot V)^2\right ) & = 
   \sum_{l=1}^N   \eta_l ^2   (t_l-t_{l-1} )E \left ( X_\infty(f_c)^2 \right )\\ 
   \nonumber
   \quad\quad & + 2  \sum_{l,j=1}^N  \eta_l  \theta_j  (t_l-t_{l-1})   E \left ( X_\infty(f_c) X_\infty (f_j)  \right)\\
   \nonumber
     \quad\quad & +  E \left (  ( \sum_{j=1}^N \theta_j X_\infty (f_j) )^2\right),
         \end{align*}
         By Lemma \ref{lem:covariance}-(3), the  first line on the right-hand side is equal to $\sum_{l=1}^N \eta_l^2 (t_l -t_{l-1}) E \left ( (Y_\infty (t_l) - Y_\infty(t_{l-1}))^2 \right)$. By Lemma \ref{lem:covariance}-(2), the second line is equal to $  2  \sum_{l,j=1}^N  \eta_l  \theta_j    E \left ( (Y_\infty (t_l) - Y_\infty(t_{l-1}) )X_\infty (f_j)  \right)$. Hence, \eqref{eq:charfun} follows. 
  \end{proof}
We need the following technical lemma, whose proof is left to the  appendix. 
\begin{lemma} 
\label{lem:product} 
Let $-\infty < a_j\le b_j <\infty,~j=1,2$ and let $(D_j,{\cal D}_j)$ denote the measure spaces given by $D ([a_j,b_j])$ equipped with the Borel $\sigma$-algebra. Suppose that $P,P_1,\dots$  is a sequence of (Borel) probability measures on $(D_1\times D_2, {\cal D}_1\times {\cal D}_2)$, satisfying 
\begin{enumerate} 
\item $(P_n:n\in\N)$ is tight. 
\item For any $N\in\N$ and $f_1,\dots,f_N\in [a_1,b_1]^N,~t_1,\dots,t_N \in [a_2,b_2]^N$ the  distribution of the marginal $ \Bigl ( \bigl(x(f_1),\dots,x(f_N)\bigr),\bigl( y(t_1),\dots,y(t_N)\bigr) \Bigr)$ under $P_n$ converges to its distribution under $P$. 
\end{enumerate}
Then $P_n \Rightarrow P$. 
\end{lemma}

We are ready to prove the main  theorem. 
 \begin{proof}[Proof of Theorem \ref{th:CLT}]
  Recall that  $\hat F_n (f) = \frac{ S_n (f) - M_n (f) }{X_n}$. 
Therefore on $[f_c,1]$, 
  \begin{align*}
 \hat \Delta_n (f) &= \frac{\ol{S}_n (f) + n (f-f_c) E I_+ }{X_n}- \frac{f-f_c}{1-f_c} -\frac{M_n(f)}{X_n}\\ & =\frac{\ol{S}_n (f)}{X_n}+ (f-f_c) \left ( \frac{E I_+ }{X_n/n} - \frac{1}{1-f_c}\right)-\frac{M_n(f)}{X_n }.
    \end{align*} 
     Since 
       \begin{align*} (f-f_c) \left ( \frac{E I_+ }{X_n/n} - \frac{1}{1-f_c}\right)&= 
        \frac{f-f_c}{1-f_c} \frac{n}{X_n}  \frac{ (1-f_c) n E I_+ - X_n}{n}\\
        &= \frac{f_c-f}{1-f_c} \frac{n}{X_n}\left ( \frac{\ol{S}_n (1) +M_n (1)}{n}\right), 
        \end{align*}
          We can write 
       $$  \sqrt{n}  \hat \Delta_n (f)   =
         \frac{n}{X_n} \left ( \frac{\ol{S}_n (f)}{\sqrt{n}} + \frac{f_c-f}{1-f_c} \frac{\ol{S}_n(1)}{\sqrt{n}}\right) + E_n(f),
         $$
          where 
          $$ E_n(f) : =\frac{n}{X_n} \left (  \frac{f_c-f}{1-f_c} \frac{ M_n (1)}{\sqrt{n}}- \frac{M_n(f)}{\sqrt{n}}\right).$$
          Fix $f_0\in (f_c,1]$.  Then,  
          $$ \sup_{f\in [f_0,1]} |E_n(f)| \le \frac{2 n}{X_n} \frac{\lim_{k\to\infty} |M_k(f_0)|}{\sqrt{n}}\to 0\mbox{ a.s.}$$
           As for $f=f_c$, observe that 
    $$ \sqrt{n} \hat F_n (f_c) =\frac{(S_n (f_c) - M_n (f_c))/\sqrt{n} }{X_n/n } = 
     \frac{( Y_n (1) - \inf_{t \le 1} Y_n(t))/\sqrt{n} }{X_n/n }=\frac{ \Psi(Y_n/\sqrt{n})}{X_n/n}.$$
      Let $T:D ([f_0,1]) \to D([f_0,1])$ be the measurable   mapping $(T x)(f):= x(f) + \frac{f_c-f}{1-f_c} x(1)$. Then we have 
       $$ \sqrt{n} \left (   \hat \Delta_n |_{[f_0,1]} ,\hat F_n (f_c) \right) 
       = \left ( \frac{n}{X_n}T ( \frac{\ol{S}_n }{\sqrt{n}}) + E_n , \frac{n}{X_n}\Psi (\frac{Y_n}{\sqrt{n}}) \right).$$ 
        By Lemma \ref{lem:tight}, the laws  of $\ol{S}_n|_{[f_0,1]}/\sqrt{n}$ are tight in $D([f_0,1])$. By the Invariance Principle \cite[Theorem 14.1]{Billin},  the laws of $Y_n/\sqrt{n}$ are tight  in $D([0,1])$.  Therefore the laws of  $\frac{1}{\sqrt{n}}(\ol{S}_n|_{[f_0,1]},Y_n)$ are tight in the product space $D([f_0,1])\times D([0,1])$. Since by Lemma \ref{lem:findim} the finite dimensional distributions converge to those of $(X_\infty,Y_\infty)$, we obtain from Lemma \ref{lem:product}  that 
        $$\frac{1}{\sqrt{n}}(\ol{S}_n|_{[f_0,1]},Y_n)\Rightarrow (X_\infty,Y_\infty)$$
        in $D([f_0,1])\times D([0,1])$.   In addition, since $X_\infty$ and $Y_\infty$ are continuous processes,   it follows from the Mapping Theorem \cite[Theorem 2.7]{Billin} that $ (T (\frac{\ol{S}_n}{\sqrt{n}}),\Psi(\frac{Y_n}{\sqrt{n}}))\Rightarrow ( T (X_\infty),\Psi(Y_\infty))$ in  $D([f_0,1])\times D([0,1])$.  In addition,  $L_n/n \to (1-f_c) E(I_+)$ a.s.  and $\sup_{f\in [f_0,1]} |E_n(f)|  \to 0$, a.s., therefore if we denote $  (T (\frac{\ol{S}_n}{\sqrt{n}}),\Psi(\frac{Y_n}{\sqrt{n}}))$ by ${\cal A}_n$, the corresponding weak limit $(T(X_\infty),\Psi(Y_\infty))$ by ${\cal A}_\infty$ and denote $(n/X_n,E_n)$ by ${\cal B}_n$ and it corresponding a.s. limit $(\frac{1}{(1-f_c)E I_+},0)$ by ${\cal B}_\infty$, then for any uniformly continuous $G: (D([f_0,1])\times D([0,1]))\times \R^2\to \R$, we have
         $$ E G ({\cal A}_n,{\cal B}_n) - E G ({\cal A}_\infty,{\cal B}_\infty) =  E
         \left (  G ({\cal A}_n,{\cal B}_n) - G({\cal A}_n,{\cal B}_\infty) \right ) +   E G ({\cal A}_n,{\cal B}_\infty) - E G ({\cal A}_\infty,{\cal B}_\infty).$$ 
          Since $G$ is uniformly continuous and ${\cal B}_n \to {\cal B}_\infty$, a.s., it follows that 
           $ E
         \left (  G ({\cal A}_n,{\cal B}_n) - G({\cal A}_n,{\cal B}_\infty) \right )\to 0$. Since ${\cal B}_\infty=(\frac{1}{(1-f_c)E I_+},0)$, and ${\cal A}_n\Rightarrow {\cal A}_\infty$, we have that  $E G ({\cal A}_n,{\cal B}_\infty) - E G ({\cal A}_\infty,{\cal B}_\infty)\to 0$.  Hence, $({\cal A}_n,{\cal B}_n)\Rightarrow ({\cal A}_\infty,{\cal B}_\infty)$ in $(D([f_0,1])\times D([0,1]))\times \R^2$. By Skorohod's representation  theorem \cite[Theorem 6.7]{Billin}, that there exists a probability space and random elements on it, ${\cal A}_n':=(T_n',\Psi_n'),~ {\cal B}_n':=(n/X'_n,E_n')$  such that
         $ ({\cal A}_n',{\cal B}_n')  \eqd ({\cal A}_n,{\cal B}_n)$ and  
          $\lim_{n\to\infty} ({\cal A}_n',{\cal B}_n')$   exists a.s. in $(D([f_0,1])\times D([0,1]))\times \R^2$.
          Therefore, if $G: D ( [f_0,1]\times D ( [0,1]) \to \R$ is continuous and bounded, it follows from the bounded convergence theorem that 
          \begin{align*}
           \lim_{n\to\infty} & E G \left (  \sqrt{n} \left (   \hat \Delta_n|_{[f_0,1]},\hat F_n (f_c) \right)   \right)\\
            &= \lim_{n\to\infty}   E G \left(  \frac{n}{X_n'}T_n'  +E_n', \frac{n}{X_n'}\Psi_n')\right)= E G \left( \frac{1}{E I_+} T(X_\infty), \frac{1}{E I_+} \Psi (Y_\infty) \right),
            \end{align*}
         Finally,  
         \begin{align*}
          T(X_\infty)(f)&= X_\infty(f) + \frac{f_c-f}{1-f_c} X_\infty(1) \\
           & = 
           \sqrt{EI_+} \mbox{Br}_1(f) +  ( f+  \frac{f_c-f}{1-f_c} ) \sigma (I_+) W_2'(1) + (1+ \frac{f_c-f}{1-f_c})\sigma (I_-) W_3'(1)\\
          &= \sqrt{EI_+} \mbox{Br}_1(f) +  (1-F)(f) \sigma_2 W_2(1) ,
          \end{align*}
            completing the proof. 
 \end{proof} 
 \section*{Appendix}
 \begin{proof}[Proof of Lemma \ref{lem:product}]
 Since $D_1$ and $D_2$ are complete and separable, so is $D_1\times D_2$. In addition, by assumption $(P_n:n\in\N)$ is tight, therefore it follows from Prohorov's Theorem that $(P_n:n\in\N)$ is relatively compact. In particular, in order to complete the proof it is enough to show that whenever $P_{n_1},P_{n_2},\dots$ is a convergent subsequence, say to $Q$, then $P=Q$. Suppose we have such a sequence. For $ s \in [a_j,b_j]$, let $J_j (s) = \{ z  \in D_j : z(s_-)\ne z (s)\}$. Also let $\pi_j : D_1 \times D_2 \to D_j$ be the (continuous) coordinate mappings,   defined through $\pi_1 (x,y) := x$ and $\pi_2 (x,y) :=y$.
 Since $Q \circ \pi_j ^{-1}$ is a Borel probability measure on $D_j$ is follows (e.g.  \cite[p. 138]{Billin})  that there exists a countable set, $R_j \subset (a_j,b_j)$, such that $Q \circ \pi_j ^{-1}(J_j (s))>0$ if and only if $s \in R_j$. Let ${\cal J}_j:= [a_j,b_j]\backslash R_j$. Let  $N\in\N$ and choose $\ol{f} \in {{\cal J}_1}^N$ and $\ol{t}\in {{\cal J}_2 }^N$. For $\ol{v}:=(\ol{v}_1,\dots,\ol{v}_N)  \in({\cal J}_j) ^N$ let $h_{j,\ol{v}}: D_j \to \R^N$ be the function 
  $ h(z) := ( z (\ol{v}_1),\dots, z (\ol{v}_N))$. Then $h_{1,\ol{f}}$ is continuous  $Q\circ \pi_1^{-1}$-a.s.  and $h_{2,\ol{t}}$ is continuous $Q\circ \pi_2^{-1}$-a.s. As a result, the function $h_{\ol{f},\ol{t}}:D_1\times D_2 \to \R^N \times \R^N$, which is defined through $ h_{\ol{f},\ol{t}}(x,y)  := (h_{1,\ol{f}}(x),h_{2,\ol{t}}(y))$ is $Q$-a.s. continuous. Therefore it follows from the Mapping Theorem \cite[Theorem 2.7]{Billin}, that 
  \begin{equation}
  \label{eq:QPsame}   Q \circ h_{\ol{f},\ol{t}}^{-1}= P \circ h_{\ol{f},\ol{t}}^{-1},\mbox{ if } \ol{f}\in {{\cal J}_1}^N\mbox{ and } \ol{t}\in {{\cal J}_2}^N.
  \end{equation} 
  Let 
    \begin{align*}
     {\cal P}_j  &:= \{ h_{j,\ol{v}}^{-1}(E) : \ol{v} \in {{\cal J}_j}^N \mbox{for some }N\in\N,E \subset \R^N\mbox{ is a Borel set}\},~j=1,2.
     \end{align*}
     Then  ${\cal P}_j$  is a $\Pi$-system. In addition, it follows from  \cite[Theorem 12.5-(iii)]{Billin}, that $ \sigma ({\cal P}_j)={\cal D}_j$. 
     Let ${\cal P}:= \{ A \times B\in {\cal D}_1 \times {\cal D}_2: A \in {\cal P}_1, B \in {\cal P}_2\}$. 
      Then ${\cal P}$ is again a  $\Pi$-system, and it is easy to see that $\sigma ({\cal P} ) = {\cal D}_1\times {\cal D}_2$. Furthermore,  \eqref{eq:QPsame} guarantees that $P$ and $Q$ coincide on ${\cal P}$. It follows from the $\Pi-\Lambda$ theorem that  $Q=P$. 
      \end{proof} 
 \bibliographystyle{amsalpha}
\bibliography{mybib} 
\end{document}